\documentclass[12pt,english]{article}
\usepackage{babel}
\usepackage[cp1250]{inputenc}
\usepackage[T1]{fontenc}
\usepackage{amsfonts}
\usepackage{amsmath}
\usepackage{amsthm}
\usepackage{mathrsfs}
\usepackage{hyperref}
\usepackage{enumerate}
\usepackage{graphicx}
\usepackage{pictex}

\newcommand{\eqq}[2]{\begin{equation}  #1  \label{#2} \end{equation}    }
\newcommand*{\norm}[1]{\left\Vert{#1}\right\Vert}
\newcommand*{\abs}[1]{\left\vert{#1}\right\vert}

\newcommand*{\Om}{\Omega}
\newcommand*{\mr}{\mathbb{R}}
\newcommand*{\izi}{\int_{0}^{\infty}}

\newcommand*{\izj}{\int_{0}^{1}}
\newcommand*{\iO}{\int_{\Omega}}
\newcommand*{\p}{\partial}
\newcommand*{\poch}{\frac{d}{dt}}

\newcommand*{\dm}{D^{(\mu)}}

\newcommand*{\im}{I^{(\mu)}}
\newcommand*{\ia}{I^{\alpha}}
\newcommand*{\al}{\alpha}
\newcommand*{\uz}{u_{0}}

\newcommand*{\ma}{\mu(\alpha)}
\newcommand*{\mg}{\frac{\mu(\alpha)}{\Gamma(1-\alpha)}}
\newcommand*{\ggg}{\Gamma(\gamma)}

\newcommand*{\vf}{\varphi}
\newcommand*{\ve}{\varepsilon}

\def\Re{\operatorname {Re}}
\def\Im{\operatorname {Im}}
\def\arg{\operatorname{arg}}
\def\divv{\operatorname{div}}
\newcommand*{\supp}{\mathop{\mathrm{supp}}}

\newcommand{\hd}{\hspace{0.2cm}}
\newcommand{\no}{\noindent}
\newcommand{\m}[1]{\mbox{#1}}

\newcommand{\C}{\mathbb{C}}

\newtheorem{rem}{{\textbf {Remark}}}
\newtheorem{lem}{{\textbf {Lemma}}}
\newtheorem{prop}{{\textbf {Proposition}}}
\newtheorem{theorem}{\textbf {Theorem}}

\newcommand{\ija}{I^{1-\alpha}}

\newcommand{\ep}{\varepsilon}
\newcommand{\tamj}{(t-\tau)^{\alpha-1}}
\newcommand{\dt}{d\tau}

\newcommand{\ga}{\Gamma(\alpha)}
\newcommand{\gja}{\Gamma(1-\alpha)}
\newcommand{\jga}{\frac{1}{\ga}}
\newcommand{\jgja}{\frac{1}{\gja}}

\newcommand{\la}{\lambda}

\newcommand{\rr}{\mathbb{R}}

\newcommand{\vm}{\varphi_{m}}

\newcommand{\vp}{\varphi}

\newcommand{\da}{D^{\alpha}}

\newcommand{\ra}{\partial^{\alpha}}

\newcommand{\un}{u^{n}}

\newcommand{\io}{\int_{\Omega}}

\newcommand{\izt}{\int_{0}^{t}}
\newcommand{\izth}{\int_{0}^{t-h}}

\newcommand{\ta}{(t-\tau)^{-\alpha}}

\newcommand{\taj}{(t-\tau)^{-\alpha-1}}

\newcommand{\ld}{L^{2}(\Omega)}

\newcommand{\sun}{\sum_{k=1}^{n}}
\newcommand{\cnk}{c_{n,k}}

\newcommand{\vmx}{\vm(x)}

\newcommand{\ddt}{\frac{d}{dt}}

\newcommand{\ftt}{\frac{w(x, \tau)}{w(x,t)}}

\newcommand{\fptbt}{ |w(x,t)|^{p-2}w(x,t) }

\newcommand{\fptabt}{|w(x,\tau)|^{p-2}w(x,\tau)}

\newcommand{\lp}{L^{p}(\Omega)}
\newcommand{\nlp}[1]{\| #1 \|_{\lp}}

\newcommand{\wct}{ w(\cdot, t)}
\newcommand{\wcta}{ w(\cdot, \tau)}

\newcommand{\ddta}{\frac{d}{d\tau}}

\newcommand{\vl}{\hat{v}}
\newcommand{\gl}{\hat{g}}
\newcommand{\kl}{\hat{k}}

\newcommand{\muad}{\mu(\al) d \al}

\newcommand{\cut}{\C\setminus (-\infty, 0]}
\newcommand{\pla}{p_{\lambda}}

\newcommand{\epz}{\ep_{0}}

\newcommand{\cmk}{c_{\overline{\mu}}}
\newcommand{\plk}{p_{\overline{\lambda}}}

\newcommand{\ral}{r^{\al}}

\newcommand{\sal}{\sin(\pi \al)}

\newcommand{\ing}{\int_{\gamma}^{1-\gamma}}

\newcommand{\jpi}{\frac{1}{2\pi i }}

\newcommand{\liN}{\lim_{N\rightarrow \infty}}

\newcommand{\idN}{\int_{\delta- i N}^{\delta +i N}}
\newcommand{\idi}{\int_{\delta- i \infty}^{\delta +i \infty}}

\newcommand{\ka}{\kappa}
\newcommand{\drr}{\frac{dr}{r}}
\newcommand{\sia}{\sin(\pi \al)}
\newcommand{\tal}{t^{-\al}}

\newcommand{\vs}{\varsigma}

\begin{document}
\title{\bf Decay of solutions to parabolic-type problem with distributed order  Caputo  derivative}

\author{ Adam Kubica${}^{*}$, Katarzyna Ryszewska\footnote{Department of Mathematics and Information
Sciences, Warsaw University of Technology, ul. Koszykowa 75, 00-662 Warsaw,
Poland, E-mail addresses:
A.Kubica@mini.pw.edu.pl, K.Ryszewska@mini.pw.edu.pl}}

\maketitle

\abstract{We consider the decay of solution to fractional diffusion equation with the distributed order Caputo derivative. We assume that the elliptic operator is time-dependent and that the weight function, contained in the definition of the distributed order Caputo  derivative, is just integrable. We establish the relation between behavior of weight function near zero and the decay rate of solution.   }

\vspace{0.4cm}

\no Keywords: distributed order fractional diffusion, weak solutions, time-depended elliptic operator, temporal decay of solution.

\vspace{0.4cm}


\maketitle

\section{Introduction}

The exponential decay of solutions to parabolic problem with the elliptic operator containing only second-order terms is well known result. It is also known that if we replace time derivative by the fractional Caputo derivative $D^{\al}$, the decay of solutions is polynomial. What is more, if we make a step further and discuss equation with the distributed order Caputo derivative $\dm$, we may obtain logarithmic decay of solutions. That is why the parabolic-type equation with $\dm$ is often called the equation of ultraslow diffusion.

In this paper we ask a question, what is a factor that affects the decay rate of solutions to equation with the distributed Caputo derivative. The results presented here are mainly inspired by papers by A. Kochubei and R. Zacher  (see \cite{Kochubei}, \cite{Kochubei2}, \cite{Zacher}-\cite{ZacherStab}). We reconstruct and extend some results from those papers, but we develop them in the framework introduced in \cite{nasza} (see also \cite{KY}, \cite{KRR}).

The logarithmic decay of solutions to  the equation with $\dm$ and time-independent elliptic operator was proved in \cite{Y} under assumption that the measure $\mu$, contained in the definition of the distributed Caputo derivative, is non-negative, continuous and $\mu(0)>0$ or $\mu$ can be represented by $\mu(\al) = \mu(0) + o(\al^{\delta})$ for some $\delta > 0$ as $\al$ tends to zero. In this paper we will show that the support of $\mu$ plays essential role in asymptotics behavior of solutions.

Very important result concerning the decay of solutions to nonlocal in time problems was obtained in \cite{Zacher}. The authors consider equation
\eqq{
\poch (k*[u-\uz]) - \divv(A(t,x)Du) = 0,
}{Za}
with zero boundary conditions and the initial condition $u(0) = \uz$. The kernel $k$ is given function of type $\mathscr{PC}$  (see \cite{Zacher}). The main result, given in theorem 1.1 and corollary 1.1, claims that the $L^{2}$ norm of solution to (\ref{Za}) is estimated by initial data multiplied by the solution to ordinary equation
\eqq{
\poch (k*[u_{\lambda}-\uz])(t) + \lambda u_{\lambda}(t)= 0,
}{ordinary}
where $\lambda$ is a constant, which depends on Poincare constant and ellipticity constant. What is more, obtained estimate is sharp. As an example, the authors discussed the equation with the distributed order Caputo derivative with weight function  $\mu \equiv 1$ and obtained logarithmic estimate. In this paper we explore this problem for the distributed order derivative and we investigate the relation between the decay rate of solution to problem (\ref{ordinary}) and the weight function $\mu$.

In \cite{nasza} we showed that assuming only $\mu \in L^{1}$,  $\mu \not \equiv   0$ we may represent $\dm u$ as $\poch (k*[u-\uz])$, where $k$ is of type $\mathscr{PC}$ , thus we are able to apply results from \cite{Zacher}. However, in this paper we give more straightforward argument than the one presented in~\cite{Zacher}, which is based on approach introduced in \cite{KY}-\cite{nasza}.

According to the results obtained in \cite{Zacher}, it is natural to study first the long-time behavior of solutions to equation
\eqq{
\dm u + \lambda u = 0,  \hd  \lambda > 0.
}{ordinary2}
The formula for solution to (\ref{ordinary2}) under assumption that $\mu \in C^{2}[0,1],$ $\mu(1) > 0$ was obtained in \cite{Kochubei}. The same formula was calculated in \cite{M} assuming only $\mu \in L^{1}(\Om),$ $\mu >0$, however without rigorous proof. One of the aim of this paper is to provide a detailed proof of the formula for solution to (\ref{ordinary2}).

In the papers mentioned above, it was shown that the behavior of measure $\mu$ near zero has an influence on the asymptotics of solutions. In \cite{Kochubei} the logarithmic decay was obtained under assumption that $\mu$ is continuous and $\mu(0) > 0$. If we replace the last condition by $\mu(\al) \sim a\al^{\delta}$ as $\al$ tends to zero for some $a,\delta > 0$, then the solution decays as $(\log t)^{-1-\delta}$. Further analysis was made in \cite{Kochubei2}, where examples of faster decay of $\mu$ near zero was considered. It was proved that, they entrails faster decay of solutions, however never polynomial.

In this paper we  investigate equation (\ref{ordinary2}) assuming less regularity on $\mu$. Precisely, we state that if $\mu \in L^{1}(0,1)$, $\mu \not \equiv 0$ the equation (\ref{ordinary2}) posses the unique absolutely continuous solution. If we assume additionally that $\ma/ \al$ is integrable on $[0,1]$, then we are able to find explicit formula for solution.  What is more, we extend the ideas from \cite{Kochubei}, \cite{Kochubei2}, that is we investigate dependence between the behavior of $\mu$ near the origin and rapidity of solution decay. Finally, under the assumption that support of $\mu$ is cut off from zero, we prove polynomial decay estimates.

The paper is organized as follows. In second section we recall the  definition of fractional operators and formulate the results.
  The third section is devoted to study the equation (\ref{ordinary2}) and decay rate of its solution. In last section we prove the main result, concerning polynomial decay of solutions to parabolic-type equation with the distributed order Caputo derivative.

\section{Notation and main results}

Equation (\ref{ordinary2}) contains distributed order fractional derivative. To define this operator we need to recall the fractional integration operator  $\ia $ and the Riemann-Liouville \ fractional  \m{derivative $\ra$}
\eqq{\ia f (t) = \jga \izt \tamj  f(\tau) \dt \hd  \m{ for }\alpha>0,}{fI}
\eqq{\partial^{\alpha}f(t) =\ddt \ija f(t)=\jgja \ddt \izt \ta f(\tau)
\dt \m{ \hd for \hd  }\alpha\in (0,1).}{fRL}
 Further, by $\da$ we denote the fractional Caputo derivative
 \[
 \da f(t)= \ra [f(\cdot)- f(0)](t),
\]
where $\alpha \in (0, 1)$. Finally, for a nonnegative and measurable function $\mu: [0,1]\longrightarrow \mathbb{R}$ we define the distributed order Caputo derivative
\eqq{\dm f (t) = \int_{0}^{1}(\da f)(t) \mu(\al) d\al.}{disCap}
Our basic assumption concerning $\mu$ is as follows
\eqq{\mu \in L^{1}(0,1), \ \  \izj \mu(\al) d\al =:c_{\mu} > 0  . }{aa}

 \no First we recall the remark from \cite{nasza}.
\begin{rem}
The assumption (\ref{aa}) implies that
\eqq{
\exists \ \ 0<\gamma< \frac{1}{2} \ \  \int_{\gamma}^{1-\gamma} \ma d\al = \frac{1-\gamma}{2} c_{\mu} > 0.
}{defg}
This statement easily follows from Darboux theorem applied to function $h:[0,\frac{1}{2}]\rightarrow \mr$ defined by
\[
h(x) = \int_{x}^{1-x}\ma d\al - \frac{1-x}{2}\izj \ma d\al.
\]
\label{gamma}
\end{rem}
\no We recall notation from \cite{Kochubei}. For $f$ absolutely continuous on $[0,T]$ the distributive order derivative takes the form
\eqq{
(\dm f)(t) = (k*f')(t),
}{e1}
where
\eqq{
k(t) = \izj \mg t^{-\al}d\al
}{e2}
and $*$ here and in the whole paper denotes the convolution on $(0,\infty)$, i.e. $(k*f)(t) = \izt k(t-\tau)f(\tau)d\tau$.
Then $\dm f$ belongs to $L^{1}(0,T)$ (see (18) in \cite{nasza}).\\
Before we present the results of this paper we need to recall theorem~4 from \cite{nasza}.

\begin{theorem}
If $\mu$ satisfies (\ref{aa}), then there exists nonnegative $g\in L^{1}_{loc}[0,\infty)$ such that the operator of fractional integration $\im$, defined by the formula $ \im u =g*u$ satisfies
\eqq{(\dm \im u)(t) = u(t) \hd \m{ for } \hd  u \in L^{\infty}(0,T),}{fi1}
\eqq{(\im\dm u)(t) = u(t) - u(0) \hd \m{ for } \hd  u \in AC[0,T].}{fi2}
Furthermore, $g$ satisfies the estimate
\eqq{g(t)\leq c\max\{t^{\gamma-1},t^{-\gamma} \}}{estig}
for some positive $c$ and $\gamma\in (0,\frac{1}{2})$, which depend only on $\mu$.
\label{fint}
\end{theorem}

\no Here $\gamma$ comes from remark~\ref{gamma} and function $g$ is given by formula (see the proof of theorem~4, \cite{nasza})
\eqq{g(t) = \frac{1}{\pi}\izi e^{-rt}\frac{\izj \sin (\pi \al)r^{\al}\mu(\al)d\al}{\left(\izj \cos (\pi \al)r^{\al}\mu(\al)d\al\right)^{2}+\left(\izj \sin (\pi \al)r^{\al}\mu(\al)d\al\right)^{2}}dr.}{calka}

\no Firstly, we will focus our attention on the problem

\eqq{
 \left\{ \begin{array}{rcl}
\dm v + \lambda v &=& 0,\\
v(0)& =& v_{0}.
 \end{array} \right.
 }{ordinary3}

\no As mentioned in introduction, system (\ref{ordinary3}) was already studied in \cite{Kochubei},\cite{Y} and \cite{M} . However, our first aim is to show that integrability of $\mu$ is enough to deduce the existence of absolutely continuous solution.

\no Let us denote  $g^{[m]}*h(t)= g* \dots * g*h(t)$, where   $g$ is taken $m$-times. Then we can formulate the following proposition.

\begin{prop}
Assume that $\mu$ satisfies (\ref{aa}). Then for each $v_{0}, \la \in \mathbb{C}$ there exists the unique absolutely continuous  solution to (\ref{ordinary3}). Furthermore, it is given by the formula
\eqq{v(t)= v_{0}\sum_{k=0}^{\infty} (- \la )^{k} g^{[k]}*1 (t),   }{h2}
where the series is absolutely convergent in Sobolev space $W^{1,1}(0,T)$ for each $T>0$.
\label{zwyczajne}
\end{prop}

\no Formula  (\ref{h2}) is not convenient to study the decay of solutions, therefore we need to obtain another representation of solution to (\ref{ordinary3}). However, to that end we need one more assumption concerning $\mu$.

\begin{theorem}\label{tw1}
Assume that $\lambda > 0$, $v_{0} \in \mr$ and $\mu$ is nonnegative function satisfying (\ref{aa}) and
\eqq{\izj \frac{\mu(\al)}{\al} d \al=:\cmk<\infty.}{f1}
Then the unique absolutely continuous solution to equation  (\ref{ordinary3})  is given by the formula
 \eqq{
 v(t) = v_{0}\frac{\lambda}{\pi}\izi \frac{e^{-rt}}{r} \frac{\izj r^{\al} \sin (\pi \al) \ma d\al}{(\lambda + \izj r^{\al} \cos (\pi \al) \ma d\al )^{2} + (\izj r^{\al} \sin (\pi \al) \ma d\al)^{2}}dr
 }{postacv}
\end{theorem}

\no It should be added that the above formula for solution coincide with the one obtained in \cite{Kochubei}, however in \cite{Kochubei} author consider different assumptions concerning $\mu$. Taking advantage from representation (\ref{postacv}) we may obtain following decay result, which corresponds to ones obtained in \cite{Kochubei} and \cite{Kochubei2}.  In particular, we show that if  $\mu$ vanishes faster in zero, then the solution of (\ref{ordinary3}) decays faster in infinity.

\begin{prop}
Assume that (\ref{aa}) and (\ref{f1}) holds. If $v$ is the solution to (\ref{ordinary3}) with $\la >0$, then for $t$ large enough the following estimate holds
\eqq{|v(t)|\leq \frac{c_{0}}{\ln{t}},}{f10}
where $c_{0}$ depends only on $v_{0}$, $\la$,  $c_{\mu}$ and $\cmk$. Furthermore, if $\vs$ is some fixed number from the interval $(0,1)$, then
\eqq{\m{ if } \hd \ka,a >0 \m{ \hd and \hd }    \ma \leq a \al^{\ka} \m{ \hd  a.e. on }(0,\vs), \hd  \hd \m{ then } \hd |v(t)|\leq \frac{c}{(\ln{t})^{\ka+1}},  }{f11}
\eqq{ \m{ if \hd }  \ka,a, \beta, m >0 \m{ \hd and \hd }  \ma \leq a\al^{\ka} e^{-\frac{\beta}{\al^{m}}}\m{ \hd a.e. on }(0,\vs) , \hd  \m{ then for any }q\in (0,1) }{f12}
\eqq{|v(t)|\leq c \frac{\Gamma(\ka+1)}{(1-q)^{\ka+1} (\ln{t})^{\ka+1}} \cdot \exp\left( - m^{\frac{1}{m+1}} (1+\frac{1}{m}) (q^{m}\beta)^{\frac{1}{m+1}}(\ln{t})^{\frac{m}{m+1}}  \right) }{f13}
\eqq{\m{ if } \hd a>0 \m{ \hd and \hd } \ma \leq a\exp(-e^{\frac{1}{\al}})\m{ \hd a.e. on }(0,\vs), \hd \m{ then } \hd |v(t)| \leq \frac{c}{t^{\ln{\ln{t}}}}.   }{f14}
Finally, if $\delta \in (0,1)$, then
\eqq{\supp{\mu} \subseteq [\delta,1] \iff |v(t)|\leq \frac{c}{t^{\delta}} \hd \m { for } \hd t > 1, \hd \m{ for some positive } c.}{f15}
\label{decayode}
\end{prop}

\no The above assumption concerning the support of $\mu$ means that $\ma=0$ a.e. on~$(0,\delta)$.

In comparison with papers \cite{Kochubei} and \cite{Kochubei2} we can notice that the estimate (\ref{f10}) coincides with (2.16) obtained in theorem~2.3 from \cite{Kochubei}, however, our assumption concerning $\mu$ is different. Next, estimate (\ref{f11}) coincides with (2.17) from \cite{Kochubei} but do not require continuity of $\mu$. Results (\ref{f13}), (\ref{f14}) extend idea from theorem 1 in \cite{Kochubei2} and give more subtle relation between decay rate of $\mu$ at the origin and rapidity of decay of solution at infinity.

\no Estimate (\ref{f15}) is on special interest because it states that if $\supp \mu$ is cut off from zero then the solution to distributed order problem (\ref{ordinary3}) decays as solution to corresponding problem with Caputo derivative.

In this paper we will study the decay of solution to the fractional diffusion equation with the distributed order Caputo derivative. To be more precise, we  assume that  $\Om \subseteq \mr^{N}$ is an open and bounded set, $\p \Om \in C^{2}$ and $N \geq 2.$ We will consider the following problem
\eqq{
 \left\{ \begin{array}{lll}
\dm u = Lu & \textrm{ in } & \Om \times (0,T)=: \Om^{T}\\
 u|_{\p \Om} = 0 & \textrm{ for  } & t \in (0,T)\\
 u|_{t=0} =\uz & \textrm{ in } & \Om, \\
\end{array} \right.
}{ba}
where
\[
Lu(x,t) = \sum_{i,j=1}^{N} D_{i}(a_{i,j}(x,t)D_{j}u(x,t))+ c(x,t)u(x,t),
\]
$a_{i,j}$ are measurable,  $a_{i,j}= a_{j,i}$ and $c\in L^{\infty}(\Omega\times(0,\infty))$.  What is more, we assume that $L$ is uniformly elliptic and $c$ is nonpositive, i.e. there exist positive constants $\lambda_{1}, \lambda_{2}$ such, that
\eqq{
\lambda_{1}|\xi|^{2} \leq \sum_{i,j}^{N}a_{i,j}(x,t)\xi_{i}\xi_{j} \leq \lambda_{2}|\xi|^{2} \hd \textrm{ for } a.a  \ \ (x,t)\in  \Omega\times [0,\infty), \hd \forall \xi \in \mr^{N} \hd \m{ and } \hd c\leq 0 .
}{elipt}

In paper \cite{nasza} we proved the existence of weak and regular solutions to the problem (\ref{ba}) assuming only that  (\ref{aa}) holds. In particular, by  theorem~1 \cite{nasza} we have

\begin{theorem}\label{istnienie}
Assume that (\ref{aa}),(\ref{elipt})  hold and $u_{0}\in L^{2}(\Omega)$. Then there exists $u$ such that for any $T>0$
\[
u \in L^{2}(0,T;H^{1}_{0}(\Om)), \hd \hd  \int_{0}^{1}  I^{1-\al}[u-\uz] \mu(\al) d\al \in {}_{0}H^{1}(0,T;H^{-1}(\Om))
\]
and $u$ is a weak solution to the problem (\ref{ba}), i.e. for all $\vf \in H^{1}_{0}(\Om) $ and a.a. $t \in (0,T)$ function $u$ fulfills the equality
\[
\poch \izj  \int_{\Om} I^{1-\al}[u(x,t)-\uz(x)]\vf(x)dx \ma d\al
\]
\eqq{
+ \sum_{i,j=1}^{N} \int_{\Om}a_{i,j}(x,t)D_{j}u(x,t)D_{i}\vf(x,t)dx= \io c(x,t)u(x,t) \vp (x)dx  .}{slabadef}
\end{theorem}

\no Now we are ready to formulate the main theorem of this paper.

\begin{theorem} \label{main}
Assume that $\mu$ is nonnegative function satisfying (\ref{aa}), (\ref{f1}) and   $u$ is the weak solution to (\ref{ba}) given by theorem \ref{istnienie} for some initial value $u_{0}\in \ld$. Then, for a.a. $t$  function  $v(t):=\| u (\cdot, t)\|_{\ld}$ satisfies (\ref{f10})-(\ref{f14}).

\no Furthermore, if $\supp{\mu}\subseteq [\delta,1]$ for some positive $\delta$, then
\eqq{
\norm{u(\cdot,t)}_{L^{2}(\Om)} \leq c\norm{\uz}_{L^{2}(\Om)} t^{-\delta} \textrm{ \hd \hd for } a.a.  \hd \ \ t\geq 1
}{dd}
where $c$ depends only on  $\lambda_{1}$, $c_{\mu}$ and Poincare  constant.
\end{theorem}

It is worth to mention that the last result shows some similarity between long-time asymptotics of the solution to the problem with the distributed order Caputo derivative and the solution to corresponding problem with the multi-term fractional Caputo derivative defined by $\sum_{i=1}^{k}D^{\al_{i}}$, where $k$ is some positive finite number and $0< a_{1} < a_{2}\dots < a_{k} < 1$. The decay of solution to the last equation in case of time-independent elliptic operator was studied in  \cite{Y2}, \cite{Y3} and \cite{luchko} and there was shown that the solution decays as $t^{-a_{1}}$ (see also \cite{Lizanik}). Thus, our results states that if $\supp{ \mu} \in [\delta,1]$ then the solution to equation with the distributive order Caputo derivative decays as a solution to the multi-term fractional diffusion equation with $\al_{1} = \delta$.

The results presented in this paper can be generalized to the case $\mu$ of the following form
\[
\mu = h+ \sum_{k=1}^{K} a_{k} \delta_{\alpha_{k}},
\]
where $h$ is nonnegative integrable function on $(0,1)$ and $\delta_{\alpha_{k}}$ is Dirac delta function with support at $\alpha_{k}\in (0,1)$. We will address this issue in forthcoming paper.

\section{Solution to ODE}

In this section we present the proofs of proposition~\ref{zwyczajne}, theorem~\ref{tw1} and proposition~\ref{decayode}.

\begin{proof}[Proof of proposition~\ref{zwyczajne}. ]  By theorem~\ref{fint},  for absolutely continuous functions problem (\ref{ordinary3}) is equivalent to the following equation
\eqq{v(t)= v_{0}-\la g * v(t).}{h3}
In order to  solve (\ref{h3}) we apply to both sides operator $- \la g *$  and add $v_{0}$. Then we get
\[
v(t)= v_{0} + v_{0}(-\la)g* 1(t) + (-\la)^{2} g* g*v(t).
\]
Iterating this procedure we arrive at
\[
v(t)= v_{0}\sum_{k=0}^{n} (-\la)^{k}g^{[k]}*1(t) + (-\la)^{n}g^{[n]}*v(t).
\]
Thus, assuming that $v$ is bounded on $[0,T] $ we can pass to the limit with $n \rightarrow \infty$ and we obtain the representation (\ref{h2}) (see the proof of lemma~1 \cite{nasza} for details). Using the similar argument we deduce that the series \[
\sum_{k=0}^{\infty} (- \la )^{k} \ddt \left( g^{[k]}*1 (t) \right)
\]
converges absolutely in $L^{1}(0, T)$ for any $T>0$. Hence formula (\ref{h2}) defines absolutely continuous function on $[0,\infty)$ and by direct calculation we check that $v$ satisfies (\ref{h3}), which is equivalent to (\ref{ordinary3}).   Uniqueness follows by Gronwall-type lemma (see lemma~1, \cite{nasza}).
\end{proof}

\begin{proof}[Proof of theorem~\ref{tw1}. ] We divide the proof onto three steps.

\no \textsf{Step 1.} At the beginning we will find the candidate for the Laplace transform of solution to (\ref{ordinary3}). Let us assume that $v(t)$ is absolutely continuous and bounded solution to (\ref{ordinary3}). We apply the Laplace transform to both sides of this equation. Then we get $\vl(p)- \frac{v_{0}}{p}= - \lambda \vl (p) \cdot \gl(p)$ for $p$ such that $\Re p>0$, where by proposition~1 \cite{nasza} we have
\eqq{\gl (p) = \frac{1}{p\kl(p)} , \hd \hd   \kl(p)= \izj p^{\al-1} \mu (\al) d\al. }{e3}
Thus we get $\vl (p)\left(1+\frac{\lambda}{p \kl (p)}\right)= \frac{v_{0}}{p}$. By assumption $\lambda $ is positive, hence the expression in brackets do not vanish in the right half space, thus we obtain
\eqq{\vl(p)= v_{0} \frac{p\kl(p)}{p[p\kl(p)+\lambda] } , \hd \hd \Re p >0 .  }{e4}

\no Now, denote by $F(p) $ the right hand side of (\ref{e4}) for $v_{0}=1$. We will find an inverse Laplace transform of $F(p)$ given in a form, which is convenient in analysis of the decay rate. We will show that
\eqq{\vl (p) = v_{0}F(p) \hd \m{ for } \hd \Re p>0, \m{ \hd where } v \m{ is given by \hd } (\ref{postacv}). }{h5}
For this purpose we apply lemma~\ref{odwrotna} from appendix. We will verify that $F(p)$ satisfies  the assumptions of lemma~\ref{odwrotna}:

\no \textsf{Assumption $1$.} We choose the main branch of logarithm and for $r= |p|$, $\vp= \arg{p}$ we see that
\[
p\kl(p) = \izj p^{\al} \muad = \izj r^{\al} \cos(\al \vp ) \muad + i \izj r^{\al} \sin(\al \vp ) \muad
\]
is analytic in $\C\setminus (-\infty, 0]$. If $|\vp|\in (0, \pi)$, then $\Im p\kl(p) \not = 0$ and for $\vp=0$ we have $\Re{ p\kl(p)} >0$. Hence positivity of $\lambda$ implies that the denominator of $F$ do not vanish in $\C\setminus (-\infty, 0]$ and thus $F$ is analytic in $\C\setminus (-\infty, 0]$.

\no \textsf{Assumption $2$. } By direct calculation we obtain that $\lim\limits_{\vp \rightarrow \pi^{-}} F(te^{\pm i \vp }):=F^{\pm}(t)$ exists for each $t>0$ and then $F^{+} = \overline{F^{-}}$.

\no \textsf{Assumption $3a$. } From (24) \cite{nasza} we have
\eqq{
\left|  \izj p^{\al} \muad \right| \geq \tilde{c} |p|^{\gamma} \hd \m{ \hd for \hd } \hd  |p|\geq 1,
}{e5}
where constant $\tilde{c}$ depends only on $\mu$ and $\gamma$ comes from (\ref{defg}). Thus  we have
\eqq{
|F(p)| \leq \frac{1}{|p|} \frac{1}{\left| 1+\frac{\lambda}{\izj p^{\al} \muad} \right| }\leq \frac{2}{|p|} \hd \m{ \hd for } \hd |p| \geq \max\{1, \left(
 2\lambda/\tilde{c} \right)^{1/\gamma}\}
}{f4}
and $|F(p)| \longrightarrow 0$  as $|p|\rightarrow \infty$.

\no \textsf{Assumption $3b$. } It is enough to show that
\eqq{\lim_{|p|\rightarrow 0 } \izj p^{\al} \muad =0,}{e7}
uniformly on $|\arg(p)|<\pi$. For this purpose we fix $\ep >0$. Then there exists positive $a$ such that $\int_{0}^{a} \muad \leq \frac{\ep}{2}$. Then for $|p| \leq 1$ we have
\[
\left| \izj p^{\al} \muad \right| \leq \int_{0}^{a} \muad + \int_{a}^{1} |p|^{\al} \muad \leq \frac{\ep}{2}+ |p|^{a} \int_{a}^{1} \muad \leq \ep,
\]
if $|p|\leq \left( \frac{\ep}{2 c_{\mu}}\right)^{1/a}$ and (\ref{e7}) is proved. Denote by $\pla$ a positive number not greater than $1$, such that
\[
\left| \izj p^{\al} \muad \right| \leq \frac{\lambda}{2}, \hd \hd \m{ for \hd } p\in \cut, \hd \arg{p} \in (-\pi, \pi), \hd |p|\leq \pla.
\]
Then for $p$ as above we can estimate $|p||F(p)|$ as follows
\eqq{|p||F(p)|\leq \frac{ \left| \izj p^{\al} \muad \right|}{\left| \lambda + \izj p^{\al} \muad  \right|} \leq \frac{ \left| \izj p^{\al} \muad \right|}{\lambda -\left|  \izj p^{\al} \muad  \right|} \leq \frac{2}{\lambda} \left|  \izj p^{\al} \muad  \right|, }{e8}
and by (\ref{e7}) we may see that   $|p||F(p)| \longrightarrow 0$, if $|p| \rightarrow 0$.

\no \textsf{Assumption $4$. } We fix $\epz \in (0, \frac{\pi}{2})$ and denote $p=re^{\pm i \vp}$, where $\vp \in (\pi-\epz, \pi)$. Constant $\lambda$ is real and we have
\[
\left| \izj p^{\al} \muad + \lambda  \right|\geq \izj r^{\al} |\sin(\pm \vp \al ) | \muad \geq \int_{\gamma}^{1- \gamma} r^{\al} \sin(\vp \al ) \muad,
\]
where $\gamma$ comes from (\ref{defg}). We note that $\sin(\vp \al ) \geq \min\{\sin{(\pi- \epz)\gamma}, \sin{\pi \gamma}  \}$ for $\al \in (\gamma, 1- \gamma)$, thus we obtain
\[
\left| \izj p^{\al} \muad + \lambda  \right|\geq c_{0} \min\{r^{\gamma}, r^{1- \gamma} \},
\]
where constant $c_{0}$ depends only on $\mu$. Thus, we may notice that
\eqq{|F(p)| \leq c_{0}^{-1} \max\{r^{-\gamma}, r^{\gamma - 1} \} \izj r^{\al-1} \muad. }{f2}
We define the majorant for $F(p)$ by the formula
\eqq{a(r)= \left\{
\begin{array}{cll}
\frac{2}{\lambda} \izj r^{\al - 1} \muad, & \m{ \hd for  \hd } & r<\pla,  \\
c_{0}^{-1} \max\{r^{-\gamma}, r^{\gamma - 1} \} \izj r^{\al-1} \muad, & \m{ \hd for \hd   } & r\in [\pla,\max\{1,\left( \frac{2\lambda}{\tilde{c}} \right)^{1/\gamma}\} ].  \\
2r^{-1}, & \m{ \hd for \hd } & r>\max\{1,\left( \frac{2\lambda}{\tilde{c}} \right)^{1/\gamma}\}, \\
\end{array}
 \right.}{f3}
Using (\ref{f4}), (\ref{e8}) and (\ref{f2}) we obtain $|F(p)|\leq a(r)$, where $p=re^{\pm i \vp}$ and $\vp \in (\pi- \epz, \pi)$. We have to show that $\frac{a(r)}{1+r}$ belongs to $L^{1}(\rr_{+})$. Obviously, it holds if and only if $a(r) \in L^{1}(0, \pla)$, but we have
\[
\int_{0}^{\pla} a(r)dr = \frac{2}{\lambda} \izj \int_{0}^{\pla} r^{\al-1} dr \muad = \frac{2}{\lambda}\izj \pla^{\al} \frac{\mu(\al)}{\al} d\al= \frac{2}{\lambda} \tilde{\pla} \izj  \frac{\mu(\al)}{\al} d\al,
\]
where $\tilde{\pla}$ is some number which belongs to the interval with endpoints  $\pla $ and $1$. Using the  assumption (\ref{f1}) we deduce that $\frac{a(r)}{1+r} \in L^{1}(\rr_{+})$.

We have just shown that under assumption (\ref{aa}) and (\ref{f1}) function $F(p) = \frac{\izj p^{\al}\ma d\al}{p[\izj p^{\al}\ma d\al + \lambda]}$ satisfies assumptions of lemma \ref{odwrotna} from the appendix, thus the inverse transform of $F$ exists and $F$ is Laplace transform of the following function
\eqq{ \frac{1}{\pi}\izi e^{-rt} \Im (F^{-}(r))dr.}{at}
Direct calculations give us
\[
\Im (F^{-}(r)) = \Im \lim_{\vf\rightarrow \pi^{-}}\left[\frac{\izj r^{\al}e^{-i\vf\al}\ma d\al}{r e^{-i\vf}\left(\izj r^{\al}e^{-i\vf\al}\ma d\al + \lambda\right)}\right] =  \frac{\lambda}{r} G(r),
\]
where
\eqq{
G(r)= \frac{\izj r^{\al} \sin (\pi \al) \ma d\al}{(\lambda + \izj r^{\al} \cos (\pi \al) \ma d\al )^{2} + (\izj r^{\al} \sin (\pi \al) \ma d\al)^{2}}.
}{g1}
Thus we proved  (\ref{h5}).

\textsf{Step 2.} Now we shall show that the function $v$ given by formula (\ref{postacv}) is absolutely continuous on $[0,\infty)$. It is interesting, that for the proof of this property we do not use the assumption (\ref{f1}). Indeed, let us denote $\plk= \min\{ (\la/ 4 c_{\mu})^{1/\delta_{\la}},1\}$, where $\delta_{\la}\in (0,1)$ is small enough such that $\int_{0}^{\delta_{\la}} \muad \leq \frac{\la}{4}$. Then  for $r\in (0, \plk]$ we may write
\[
\left|  \izj r^{\al} \cos(\pi \al ) \muad \right|\leq
\int_{0}^{\delta_{\la}} \muad + \int_{\delta_{\la}}^{1} r^{\al} \muad
\]
\[
\leq \frac{\la}{4} + r^{\delta_{\la}}\int_{\delta_{\la}}^{1} \muad \leq \frac{\la}{4}+ \plk^{\delta_{\la}}c_{\mu} \leq \frac{\la}{2}.
\]
Thus we obtain
\eqq{G(r) \leq \left( \frac{2}{\la}  \right)^{2} \izj \ral \sin(\pi \al ) \muad  \hd \m{ for } \hd r\in (0, \plk]. }{g3}
Next, if $r>\plk$, then we have
\[
\izj \ral \sal \muad \geq \ing \ral \sal \muad \geq  \sin(\pi \gamma)\ing \ral \muad
\]
\[
\geq \sin(\pi \gamma) \min\{ r^{1- \gamma}, r^{\gamma}\} \ing \muad = c_{1}\min\{ r^{1- \gamma}, r^{\gamma}\},
\]
where $\gamma$ comes from (\ref{defg}) and $c_{1}$ depends only on $\mu$. Thus we obtain the estimate
\eqq{G(r) \leq c_{1}^{-1} \max\{ r^{ \gamma-1}, r^{-\gamma}\}, \hd \m{ for } \hd r>\plk.   }{g4}
We shall prove that
\eqq{\frac{1}{r}G(r) \in L^{1}(0,\infty). }{grr}
Indeed, we may write
\[
\izi \frac{1}{r}G(r) dr =  \left(
\int_{0}^{\plk}  + \int_{\plk}^{1}  + \int_{1}^{\infty} \right) \frac{1}{r} G(r) dr.
\]
Using  (\ref{g3})  we obtain the estimate
\[
\int_{0}^{\plk}\frac{1}{r} G(r) dr \leq \left( \frac{2}{\la} \right)^{2} \int_{0}^{\plk} \izj r^{\al-1}\sal \muad dr
\]
\[
 = \left( \frac{2}{\la} \right)^{2}  \izj \plk^{\al} \frac{\sal}{\al} \mu(\al) d\al
\leq \pi\left( \frac{2}{\la} \right)^{2}  \izj  \mu(\al) d\al <\infty.
\]
Applying (\ref{g4}) we get
\[
\int_{\plk}^{1} \frac{1}{r}G(r) dr \leq c_{1}^{-1}\int_{\plk}^{1} r^{\gamma-2} dr \leq \frac{c_{1}^{-1}}{1- \gamma} \plk^{\gamma-1}<\infty,
\]
and
\[
\int_{1}^{\infty} \frac{1}{r} G(r)dr \leq c_{1}^{-1} \int_{1}^{\infty} r^{-\gamma- 1} dr = \frac{c_{1}^{-1}}{\gamma}<\infty,
\]
and we obtained (\ref{grr}). Then we may apply Lebesgue monotone theorem and conclude that  $v \in C[0,\infty)$ and obviously $v \in C^{\infty}(0,\infty)$. What is more we may differentiate under the integral sign and we get
\[
v'(t)= -v_{0}\frac{\la}{\pi} \izi e^{-rt}G(r)dr
\]
and by Fubini theorem
\[
\izi |v'(t)|dt = |v_{0}|\frac{\la}{\pi} \izi \izi e^{-rt} G(r) dt dr = |v_{0}|\frac{\la}{\pi}
\int_{0}^{\infty}  \frac{1}{r} G(r) dr <\infty.
\]
Thus $v$ given by formula (\ref{postacv}) is absolutely continuous on $[0, \infty)$, provided (\ref{aa}) holds.

\textsf{Step 3. } It remains to show that $v(t)$ satisfies (\ref{ordinary3}). We will do it rigorously, but for this purpose we need the assumption (\ref{f1}). We emphasize that we can not directly verify (\ref{ordinary3}). We will show that (\ref{ordinary3}) holds by applying property (\ref{h5}), which were obtained under the assumption (\ref{f1}).

We begin with the initial condition. We have to show that
\eqq{\frac{\lambda}{\pi}\izi  \frac{\izj r^{\al} \sin (\pi \al) \ma d\al}{(\lambda + \izj r^{\al} \cos (\pi \al) \ma d\al )^{2} + (\izj r^{\al} \sin (\pi \al) \ma d\al)^{2}} \frac{dr}{r}=1,}{f9}
where $\la $ is arbitrary positive number and  $\mu $ satisfies (\ref{aa}), (\ref{f1}). For this purpose we  use  formula for the inverse Laplace transform
\[
v(t)= \jpi \liN \idN e^{pt} v_{0} F(p) dp,
\]
where $\delta$ is positive and we will pass to the  limit $t\rightarrow 0$. We may write
\[
v(t)= \jpi \liN \idN e^{pt } \frac{v_{0}}{p} dp - \frac{ \la v_{0}}{2\pi i } \liN \idN e^{pt} \frac{1}{p[p \kl(p)+\la]} dp,
\]
and the first expression is equal to $v_{0}$, because for the  Laplace transform we have \m{$\mathcal{L}[1](p)=\frac{1}{p}$}.  It remains to show that  the second expression vanishes at $t=0$. Indeed, we note that
\[
\liN \idN e^{pt} \frac{1}{p[p \kl(p)+\la]} dp=  \idi (e^{pt}-1) \frac{1}{p[p \kl(p)+\la]} dp+  \idi  \frac{1}{p[p \kl(p)+\la]} dp\equiv A(t)+B.
\]
We may notice that both integrals converge absolutely.  Indeed, for $p=\delta + i s$ and $\gamma$ from (\ref{defg}) we may estimate from below
\[
|p[p\kl (p)+\la]|\geq \izj |\delta^{2} + s^{2}|^{\frac{\al+1}{2}} \sin\left(\al \arctan{\frac{|s|}{\delta}} \right) \muad
\]
\[
\geq \ing |\delta^{2} + s^{2}|^{\frac{\al+1}{2}} \sin\left(\al \arctan{\frac{|s|}{\delta}} \right) \muad
\]
and for $|s|\geq \max\{1, \delta \}$ we arrive at the following estimate
\[
|p[p\kl (p)+\la]|\geq  |s|^{\gamma+1} \sin(\frac{\pi}{4}\gamma) \ing \muad.
\]
Hence $A(t)$ and $B$ converge absolutely. Next, applying Lebesgue dominated convergence theorem we deduce that $A(0)=0$.

\no In order to show that $B=0$ we apply Cauchy theorem and we obtain that
\eqq{\idN \frac{1}{p[p\kl(p)+\la ]}dp = \int_{\Gamma_{N} } \frac{1}{p[p\kl(p)+\la ]}dp , }{f6}
where $\Gamma_{N}= \{\delta +  Ne^{i\vp}: \hd \vp \in (- \frac{\pi}{2}, \frac{\pi}{2}) \}$. We will estimate the last integral
\[
\left|\int_{\Gamma_{N} } \frac{1}{p[p\kl(p)+\la ]}dp \right| \leq \int_{- \frac{\pi}{2}}^{\frac{\pi}{2}} \frac{1}{|(\delta+ Ne^{i\vp})\kl (\delta+ Ne^{i\vp}) + \la  |} d \vp
\]
We may calculate that
\[
|(\delta+ Ne^{i\vp})\kl (\delta+ Ne^{i\vp}) + \la  |^{2}
  \]
\[
 =\left[ \izj |\delta^{2}+ 2\delta N \cos{\vp} + N^{2}|^{\frac{\alpha}{2}}   \cos\left( \al \arctan\left( \frac{N \sin\vp}{N\cos{\vp}+\delta}  \right)\right) \muad  +\la \right]^{2}
\]
\[
 +\left[ \izj |\delta^{2}+ 2\delta N \cos{\vp} + N^{2}|^{\frac{\alpha}{2}}   \sin\left( \al \arctan\left( \frac{N | \sin\vp|}{N\cos{\vp}+\delta}  \right)\right) \muad   \right]^{2}\equiv B^{2}_{1}(\vp)+ B^{2}_{2}(\vp).
\]
For $|\vp | \in (\frac{\pi}{4}, \frac{\pi}{2})$ we can estimate $B_{2}(\vp )$ as follows
\[
B_{2}(\vp ) \geq \ing |\delta^{2}+ 2\delta N \cos{\vp} + N^{2}|^{\frac{\alpha}{2}}   \sin\left( \al \arctan\left( \frac{N | \sin\vp|}{N\cos{\vp}+\delta}  \right)\right) \muad
\]
\[
\geq \ing |\delta^{2}+ 2\delta N \cos{\vp} + N^{2}|^{\frac{\alpha}{2}}   \sin\left( \gamma \arctan\left( \frac{N}{N+\sqrt{2}\delta}  \right)\right) \muad
\]
\[
\geq N^{\gamma}    \sin\left( \gamma \arctan\left( \frac{N}{N+\sqrt{2}\delta}  \right)\right) \ing\muad.
\]
Hence for  $|\vp | \in (\frac{\pi}{4}, \frac{\pi}{2})$ and $N \geq \max\{ 1, \frac{\sqrt{2} \delta}{ \sqrt{3}-1} \}  $ we have
\eqq{B_{2}(\vp) \geq N^{\gamma} \sin\left(\frac{\pi}{6}\gamma \right) \ing \muad.}{f7}

\no Similarly, for $\vp  \in [-\frac{\pi}{4}, \frac{\pi}{4}]$ and $B_{1}(\vp )$ we obtain
\[
B_{1}(\vp ) \geq \ing |\delta^{2}+ 2\delta N \cos{\vp} + N^{2}|^{\frac{\alpha}{2}}   \cos\left( \al \arctan\left( \frac{N  \sin\vp}{N\cos{\vp}+\delta}  \right)\right) \muad
\]
\[
\geq \ing |\delta^{2}+ 2\delta N \cos{\vp} + N^{2}|^{\frac{\alpha}{2}}   \cos\left( (1-\gamma) \arctan\left( \frac{N}{N+\sqrt{2}\delta}  \right)\right) \muad
\]
\[
\geq N^{\gamma}    \cos\left( (1-\gamma) \arctan\left( \frac{N}{N+\sqrt{2}\delta}  \right)\right) \ing\muad.
\]
Hence for  $\vp  \in [-\frac{\pi}{4}, \frac{\pi}{4}]$ and $N \geq 1 $ we have
\eqq{B_{1}(\vp) \geq N^{\gamma} \cos\left((1-\gamma)\frac{\pi}{4} \right) \ing \muad.}{f8}
Therefore from (\ref{f7}) and (\ref{f8}) we deduce that
\[
|(\delta+ Ne^{i\vp})\kl (\delta+ Ne^{i\vp}) + \la  | \geq c(\mu) N^{\gamma},
\]
for $N $ large enough and taking the limit $N\rightarrow \infty$ in (\ref{f6}) we obtain that $B=0$. Thus we proved that $v$ given by formula (\ref{postacv}) satisfies initial condition $v(0)=v_{0}$ and in particular we calculated the integral (\ref{f9}).

Now we shall show that $v$ satisfies equation (\ref{ordinary3})${}_{1}$. We denote by $w(t)= \dm v(t)+ \la v(t)$. We already proved that $v$ is absolutely continuous, hence $w$ belongs to $L^{1}_{loc}[0, \infty)$. Using  the fractional integration operator $\im = g*$ and applying theorem~\ref{fint} we obtain
\[
g* w(t)= v(t)-v_{0}+ \la g* v(t).
\]
The Laplace transform of right-hand side is equal to zero, because  (\ref{e3}) and (\ref{e4}) hold. Then $g*w(t)=0$ for a.a. $t>0$. We will show directly that $w\equiv 0$.  Indeed, if we denote by $\eta_{\ep}(t)= \ep^{-1}\chi_{(0,\ep)}(t)$, then $\eta_{\ep}* g* w(t)=0$ for a.a. $t>0$ and $\eta_{\ep}*w$ is in $L^{\infty}(0,T)$ for each $T>0$. Then we may apply theorem~\ref{fint}  and we have $\dm g* \eta_{\ep}* w(t)=\eta_{\ep}* w(t)=0$ for a.a. $t\in (0,T)$. On the other hand, from Lebesgue differentiation theorem we have $\eta_{\ep}* w(t) \rightarrow w(t)$, a.e. as $\ep\rightarrow 0$ and we deduce that  $w\equiv0$.
\end{proof}

\no Now, we are ready to show the decay estimates for solution to (\ref{ordinary3}), that is the proof of proposition~\ref{decayode}.

\begin{proof}[Proof of proposition~\ref{decayode}]
From theorem~\ref{tw1}, using notation (\ref{g1})  we obtain following formula for solution to equation (\ref{ordinary3})
\eqq{
v(t)= v_{0}\frac{\la}{\pi} \izi e^{-rt} G(r)\drr.
}{ff1}
Thus using (\ref{g3}) we may estimate the absolute value of $v$ by
\[
|v(t)| \leq |v_{0}| \frac{4}{\pi \la } \int^{\plk}_{0} e^{-rt} \izj r^{\al} \sia \muad \drr + |v_{0}| \frac{\la}{\pi} \int_{\plk}^{\infty} e^{-rt} G(r) \drr .
\]
The last integral decays exponentially, because due to (\ref{g4}) we obtain
\[
\int_{\plk}^{\infty} e^{-rt} G(r) \drr \leq c_{1}^{-1} \plk^{-\gamma-1} \frac{e^{-\plk t}}{t}.
\]
Thus it is enough to estimate
\eqq{A:=\int^{\plk}_{0} e^{-rt} \izj r^{\al} \sia \muad \drr.}{f16}
Applying Fubini theorem and substituting $s=rt$ we obtain the equality
\[
A= \izj t^{-\al} \int_{0}^{\plk t } e^{-s} s^{\la -1 } ds \sia \muad \leq \izj t^{-\al} \Gamma(\al) \sia \muad = k(t),
\]
where $k$ was defined in (\ref{e2}). Using (\ref{f1}) we get
\[
k(t) \leq \max_{\al \in [0,1]}{\al t^{-\al}} \cdot \izj \frac{\ma}{\al} d\al= \frac{\cmk}{e\ln{t}}, \hd \m{ for } \hd t>e,
\]
thus we proved (\ref{f10}).

\no Now assume that $\vs \in (0,1)$. If  $\ma \leq a \al^{\ka}$ a.e. on $(0,\vs)$, then substituting $s=\alpha \ln t$  for $t\geq 0$ we get
\[
k(t) \leq \int_{0}^{\vs} t^{-\alpha} \muad +\int_{\vs}^{1} t^{-\alpha} \muad
\leq a \izj t^{-\al} \al^{k} d\al + c_{\mu}t^{-\vs}
\]
\[
= \frac{a}{(\ln{t})^{\ka+1}} \int_{0}^{\ln{t}} e^{-s} s^{\ka} ds  + c_{\mu}t^{-\vs} \leq a\frac{\Gamma(\ka+1)}{(\ln{t})^{\ka+1}} + c_{\mu}t^{-\vs},  \hd \m{ for } \hd t>1,
\]
hence we obtain (\ref{f11}).

\no Analogously,  in the case (\ref{f12}) we take $q\in (0,1 )$ and we have
\[
k(t) \leq a\izj \tal \al^{\ka} e^{-\frac{\beta}{\al^{m}}} d\al  + c_{\mu}t^{-\vs}.
\]
To estimate the first term we write
\[
a\izj t^{-q\al}  e^{-\frac{\beta}{\al^{m}}} \cdot t^{-(1-q)\al}\al^{\ka} d\al
\leq a \max_{\al\in [0,1]}  t^{-q\al}  e^{-\frac{\beta}{\al^{m}}} \izj t^{-(1-q)\al}\al^{\ka} d\al.
\]
Proceeding as in the previous case we get
\[
\izj t^{-(1-q)\al}\al^{\ka} d\al \leq \frac{\Gamma(\ka+1)}{(1-q)^{\ka+1}(\ln{t})^{\ka+1}}.
\]
By direct calculation we obtain
\[
\max_{\al\in [0,1]}  t^{-q\al}  e^{-\frac{\beta}{\al^{m}}}  = \exp\left( - m^{\frac{1}{m+1}} (1+\frac{1}{m}) (q^{m}\beta)^{\frac{1}{m+1}}(\ln{t})^{\frac{m}{m+1}}   \right) \hd \m{ for } \hd t>e^{\frac{\beta m }{q}},
\]
hence (\ref{f13}) is proved.

\no Now assume that $\ma \leq a \exp(-\exp( 1/ \al  ))$ a.e. on $(0,\vs)$. Then we have
\[
k(t) \leq a\izj \tal \exp(-\exp(1/\al)) d\al+   c_{\mu}t^{-\vs}.
\]
To estimate the integral we take  $b\in (0,1)$
\[
 \left( \int_{0}^{b} + \int_{b}^{1} \right)\tal \exp(-\exp(1/\al)) d\al
\leq  \exp(-\exp(1/b)) \int_{0}^{b} \tal  d\al  + \int^{1}_{b} \tal  d\al.
\]
For $t>e^{e}$ we set  $b=(\ln{\ln{t}})^{-1}$ and we arrive at
\[
k(t) \leq at^{-\frac{1}{\ln{\ln{t}}}}+c_{\mu}t^{-\vs},
\]
thus we obtained (\ref{f14}).

\no In order to show (\ref{f15}) we first notice that if $\supp \mu \subseteq [\delta,1]$, then
\[
k(t) = \int_{\delta}^{1}\frac{t^{-\al}}{\Gamma(1-\al)}\ma d\al \leq c_{\mu} t^{-\delta} \hd \m{ \hd for \hd  } t \geq 1.
\]
It remains to show that polynomial decay of $v$ with ratio  $t^{-\delta}$ implies that  $\supp \mu \subseteq [\delta,1]$. We  suppose contrary. Then there exist $0< a < b < \delta$ such that $\int_{a}^{b}\ma d\al > 0.$ We will estimate from below the absolute value of $v$. Function $G(r)$ is positive, hence
\[
|v(t)| \geq \frac{|v_{0}|}{\lambda \pi}  \izj \frac{e^{-rt}}{r} G(r)dr.
\]
To estimate $G(r)$ we  first notice that for $r \leq 1$
\[
\izj r^{\al} \sin (\pi \al) \ma d\al \geq \int_{a}^{b} r^{\al} \sin (\pi \al) \ma d\al \geq  c r^{b},
\]
where $c$ depends only on $a$, $b$ and $c_{\mu}$. Furthermore, for $r\leq 1$
\[
\left(\lambda + \izj r^{\al}\cos (\pi\al)\ma d\al\right)^{2} + \left(\izj r^{\al}\sin (\pi\al)\ma d\al\right)^{2} \leq (\lambda + c_{\mu})^{2} + c_{\mu}^{2}.
\]
Thus,  having in mind that $G$ was defined by the formula (\ref{g1}) we obtain that $G(r) \geq c_{1}r^{b}$ for $r\in (0,1)$. Thus,   for some positive $C$ and $t>1$ we have
\[
|v(t)| \geq C \izj e^{-rt} r^{b-1}dr =C t^{-b}\izt e^{-s} s^{b-1}ds \geq C t^{-b}\izj  e^{-s}s^{b-1}ds .
\]
Then if $|v(t)| \leq c_{0}t^{-\delta}$ for some $c_{0}$ and $t\geq 1$, we get a contradiction because $b<\delta$.
\end{proof}

\section{Proof of the main result}

In this section we present a proof of theorem~\ref{main}.

\begin{proof}[Proof of theorem~\ref{main}]
We shall obtain the additional estimate for the decay of  the sequence of approximated solutions of (\ref{ba}) obtained in \cite{nasza}. Namely, let
\[
\un (x, t) = \sun \cnk (t) \vmx,
\]
be approximated solution constructed in the proof of theorem~1 in \cite{nasza}. The coefficients $\cnk$ are uniquely determined by the system (36) in \cite{nasza}, which in the case of problem (\ref{ba}) reduces to
\[
 \int_{\Om} \dm \un (x,t) \cdot  \vf_{m}(x)dx
\]
\eqq{= -\sum_{i,j=1}^{N}\int_{\Om} a^{n}_{i,j}(x,t)D_{j} \un (x,t) \cdot D_{i}\vf_{m}(x)dx +\io c^{n}(x,t)u^{n}(x,t)\vm(x)dx.
}{de}
Here $\{\vf_{n}\}_{n=1}^{\infty}$ form an orthonormal basis of $L^{2}(\Om) $ and are defined as eigenfunctions of the Laplace operator
\begin{equation}\label{vf}
 \left\{ \begin{array}{rlll}
-\Delta \vf_{n} &=& \lambda_{n}\vf_{n} & \textrm{ in } \hd  \Om \\
 {\vf_{n}}_{| \p \Om } &=& 0. &  \\
\end{array} \right. \end{equation}
Furthermore, $a^{n}_{i,j}$ and $c^{n}$  are defined as follows: let  $\eta_{\ve}= \eta_{\ve}(t)$ be a standard smoothing kernel with the support in $[-\ve,\ve]$ and we set
\[
a_{i,j}^{n}(x,t) = \eta_{\frac{1}{n}}(\cdot)\overline{*}a_{i,j}(x,\cdot)(t), \hd \hd c^{n}(x,t)= \eta_{\frac{1}{n}}(\cdot)\overline{*}c(x,\cdot)(t),
\]
where $\overline{*}$ denotes convolution on real line and we extended  $a_{i,j}(x,t)$ by even reflection for $t <0$, but  $c$ we extend by zero for $t<0$. Now we fix $T$ and we shall prove estimate for $u^{n}$ with constant independent on $T$. First we notice that  $a_{i,j}^{n}\rightarrow a_{i,j}$, \hd $c^{n} \rightarrow c$ in $L^{2}(\Om^{T})$ and
\eqq{
\lambda|\xi|^{2} \leq \sum_{i,j}^{N}a_{i,j}^{n}(x,t)\xi_{i}\xi_{j} \leq \Lambda|\xi|^{2} \textrm{  } \hd \forall t \in [0,T], \hd \forall \xi \in \mr^{N}, \m{ a.a. } x\in \Omega, \hd \m{ and } \hd c^{n} \leq 0.
}{elipn}
We recall that in \cite{nasza} we constructed approximate solutions $u^{n}$ in such a way that
\eqq{ u^{n}(x,\cdot) \in C([0,T]), \hd  \ t^{1-\gamma}\un_{t}(x, \cdot) \in C([0,T]) \textrm{ for each } x\in \overline{\Omega},
}{dp}
where $\gamma$ is from (\ref{defg}). In order to obtain necessary estimates we introduce $u^{n}_{\ve}(x,t) = u^{n}(x,t) + \ve \vf_{n+1}(x)$, where $\ve > 0$ (see the proof of theorem~1.1 in \cite{Zacher}). Then for all $t \geq 0$ we have $\norm{u^{n}_{\ve}(\cdot,t)}_{L^{2}(\Om)} \geq \ve > 0.$
We multiply (\ref{de}) by $c_{n,k}$ and sum it up for $m=1,\dots,n$. Then using $\dm u^{n}_{\ve}  =\dm u^{n} $ we get
\[
\int_{\Om} \dm u^{n}_{\ve} (x,t) \cdot u^{n}(x,t)dx =
\]
\eqq{
 -\sum_{i,j=1}^{N}\int_{\Om} a^{n}_{i,j}(x,t)D_{j} \un (x,t) \cdot D_{i}u^{n}(x,t)dx + \io c^{n}(x,t)|u^{n}(x,t)|^{2}dx.
}{df}
From the definition of $u^{n}_{\ve}$ and orthogonality of $\{\vf_{n}\}_{n=1}^{\infty}$ we obtain following equality
\[
\int_{\Om} \dm u^{n}_{\ve} (x,t) \cdot u^{n}(x,t)dx
= \int_{\Om} \dm u^{n}_{\ve} (x,t) \cdot u^{n}_{\ve}(x,t)dx
\]
Using this equality and   (\ref{elipn}) we get
\eqq{
\int_{\Om} \dm u^{n}_{\ve} (x,t) \cdot u^{n}_{\ve}(x,t)dx + \lambda \norm{D u^{n}(\cdot,t)}^{2}_{L^{2}(\Om)} \leq 0.
}{dg}
Function $u^{n}_{\ve}$ satisfies assumptions of  proposition~\ref{Awniosek} in the appendix, hence making use of inequality (\ref{a1})  and denoting Poincare constant by $c_{p}$ we obtain
\eqq{\norm{u^{n}_{\ve}(\cdot,t)}_{L^{2}(\Om)}\dm \norm{u^{n}_{\ve}(\cdot,t)}_{L^{2}(\Om)} + \frac{\lambda}{c_{p}^{2}} \norm{u^{n}(\cdot,t)}_{L^{2}(\Om)}^{2} \leq 0 }{dh}
Using again the definition of $u^{n}_{\ve}$ we may write that
\eqq{\norm{u^{n}_{\ve}(\cdot,t)}_{L^{2}(\Om)}\dm \norm{u^{n}_{\ve}(\cdot,t)}_{L^{2}(\Om)} + \frac{\lambda}{c_{p}^{2}} \norm{u^{n}_{\ve}(\cdot,t)}_{L^{2}(\Om)}^{2} \leq \frac{\lambda \ve^{2}}{c_{p}^{2}} }{di}
Having in mind that $\norm{u^{n}_{\ve}(\cdot,t)}_{L^{2}(\Om)} \geq \ve > 0$ we may divide by $\norm{u^{n}_{\ve}(\cdot,t)}_{L^{2}(\Om)}$ and we get that
\eqq{\dm \norm{u^{n}_{\ve}(\cdot,t)}_{L^{2}(\Om)} + \frac{\lambda}{c_{p}^{2}} \norm{u^{n}_{\ve}(\cdot,t)}_{L^{2}(\Om)} \leq \frac{\lambda \ve}{c_{p}^{2}} }{dii}

\no We will pass to the limit with $\ve$. From the definition of $u^{n}_{\ve}$ we have
\eqq{
\norm{u^{n}_{\ve}(\cdot,t)}_{L^{2}(\Om)} \longrightarrow \norm{u^{n}(\cdot,t)}_{L^{2}(\Om)}, \textrm{ as \hd  } \ve \rightarrow 0,}{zbl2}
for every $t\in  [0,T]$. We will show that for every $t \in [0,T]$
\eqq{
\dm \norm{u^{n}_{\ve}(\cdot,t)}_{L^{2}(\Om)} \rightarrow \dm \norm{u^{n}(\cdot,t)}_{L^{2}(\Om)} \textrm{ as } \ve \rightarrow 0.}{dl}
Indeed, we set
\[
T_{1} = \{ t: \norm{u^{n}(t)}_{L^{2}(\Om)} > 0\}\textrm{ \hd and \hd  }T_{2} = [0,T] \setminus T_{1}.
\]
We note that $\norm{u^{n}(\cdot,t)}_{L^{2}(\Om)} \in AC[0,T]$ as superposition of Lipschitz function with absolutely continuous function. Thus, we can calculate its derivative as follows
\eqq{
\poch \norm{u^{n}(\cdot,t)}_{L^{2}(\Om)} =  \left\{ \begin{array}{ll}
\frac{\iO u^{n}(x,t) u^{n}_{t}(x,t)dx}{\norm{u^{n}(\cdot,t)}_{L^{2}(\Om)}} & \textrm{for }a.a. \ \ t \in T_{1} \\
0 & \textrm{for }a.a. \ \ t \in T_{2}.\\
\end{array} \right.
}{do}
Using the orthogonality of $\{\vm\}$ in $\ld$ and definition of $T_{2}$ we may write
\[
\dm \norm{u^{n}_{\ve}(\cdot,t)}_{L^{2}(\Om)} = \izj \mg \int_{T_{1}} (t-\tau)^{-\al}\frac{\iO u^{n}(x,\tau) u^{n}_{t}(x,\tau)dx}{\left(\ep^{2}+\iO \abs{u^{n}(x,\tau)}^{2}dx\right)^{\frac{1}{2}}}d\tau d\al \equiv I.
\]
To take the limit under integral sing we need the estimate. By Schwarz inequality and property (\ref{dp}) we have
\[
I \leq \izj \mg  \int_{T_{1}} (t-\tau)^{-\al} \left(\iO\abs{u^{n}_{t}(x,\tau)}^{2}dx\right)^{\frac{1}{2}}d\tau d\al
\]
\[
 \leq |\Om| \norm{t^{1-\gamma}u^{n}_{t}(x,t)}_{L^{\infty}(\Om^{T})}\izj \mg \izt (t-\tau)^{-\al} \tau^{\gamma-1} d\tau d\al
  \]
  \[
  = |\Om| \norm{t^{1-\gamma}u^{n}_{t}(x,t)}_{L^{\infty}(\Om^{T})} \izj t^{\gamma-\al} \frac{\Gamma(\gamma)}{\Gamma(1+\gamma-\al)} \ma d\al
   \]
   \[
   \leq 2 |\Om| \ggg  c_{\mu} \norm{t^{1-\gamma}u^{n}_{t}(x,t)}_{L^{\infty}(\Om^{T})}\max\{t^{\gamma},t^{\gamma-1}\}.
\]
Thus from Lebesgue dominated theorem we obtain that as $\ve$ tends to zero
\[
 \forall t \in [0,T] \hd \hd  \dm  \norm{u^{n}_{\ve}(\cdot,t)}_{L^{2}(\Om)} \longrightarrow \izj \mg \int_{T_{1}} (t-\tau)^{-\al}\frac{\iO u^{n}(x,\tau) u^{n}_{t}(x,\tau)dx}{\left(\iO \abs{u^{n}(x,\tau)}^{2}dx\right)^{\frac{1}{2}}}d\tau d\al
\]
\eqq{
 =\izj \mg \int_{T} (t-\tau)^{-\al} \frac{d}{d\tau} \| u^{n}(\cdot, \tau)\|_{\ld}  d\tau d\al = \dm  \norm{u^{n}(\cdot,t)}_{L^{2}(\Om)},
}{dj}
where in the first equality we applied (\ref{do}). And that way we proved (\ref{dl}). Thus, using (\ref{zbl2}) and (\ref{dl}) in (\ref{dii}) we obtain the following inequality
\eqq{
\dm \norm{u^{n}(\cdot,t)}_{L^{2}(\Om)} + \frac{\lambda}{c_{p}^{2}} \norm{u^{n}(\cdot,t)}_{L^{2}(\Om)} \leq 0, \hd \hd t\in [0,T].
}{dm}
From theorem~\ref{tw1} we know that there exists uniquely defined absolutely continuous solution to the problem
\eqq{\dm v_{n}(t)+ \frac{\lambda}{c_{p}^{2}} v_{n}(t)=0, \hd \hd v_{n}(0)=\norm{u^{n}(\cdot,0)}_{L^{2}(\Om)}.}{bc}
We set $f(t):=\norm{u^{n}(\cdot,t)}_{L^{2}(\Om)}- v_{n}(t)$. Then
\eqq{\dm f(t)+ \frac{\lambda}{c_{p}^{2}} f(t)\leq 0, \hd \hd t\in [0,T], \hd \hd f(0)=0 .}{f19}
We shall show that $f$ is non-positive, i.e.
\eqq{\norm{u^{n}(\cdot,t)}_{L^{2}(\Om)}\leq v_{n}(t), \hd \hd \m{ for } t\in [0,T].}{f20}
Indeed, in opposite case there exists $t_{0}\in (0,T]$ such that $f(t_{0})= \max\limits_{t\in [0,T]}f(t)>0$. Then from (\ref{postacv}) and (\ref{dp}) we deduce that $f\in W^{1, \infty}(\ka , T)$ for each $\ka >0$. Thus by lemma~\ref{luchko} in the appendix  we have $(\dm f)(t_{0})\geq 0$, which is a contradiction with (\ref{f19}) and we obtained (\ref{f20}). By Bessel inequality and (\ref{postacv}) we have $v_{n}\leq v$, where $v$ is a solution to the problem
\eqq{\dm v(t)+ \frac{\lambda}{c_{p}^{2}} v(t)=0, \hd \hd v(0)=\norm{u_{0}}_{L^{2}(\Om)}.}{bcc}
Further, if $\eta\in C^{\infty}_{0}(0,T)$ is non-negative, then from (\ref{f20}) we have
\eqq{\int_{0}^{T}  \eta(t)\| u^{n}(\cdot, t)\|_{\ld}^{2} dt \leq \int_{0}^{T} \eta(t) v^{2}(t) dt.   }{f21}
There exists a subsequence $\{u^{n} \}$ (still indexed by n) such, that $u^{n}\rightharpoonup u$ in $L^{2}(0,T;H^{1}_{0}(\Omega))$, where $u$ is the unique solution to (\ref{ba}) given by theorem~\ref{istnienie} (see (46) in \cite{nasza}). Thus by weak lower semi-continuity of the $L^{2}$ norm we have
\[
\int_{0}^{T}  \eta(t)\| u(\cdot, t)\|_{\ld}^{2} dt \leq \int_{0}^{T} \eta(t) v^{2}(t) dt,
\]
and as a consequence we obtain
\[
\| u(\cdot, t)\|_{\ld}\leq v(t) \hd \m{for a.a. \hd } t\in [0,T].
\]
Here $T>0$ is arbitrary and by theorem~\ref{istnienie} function $u$ is uniquely defined on $[0,\infty)$, thus we get  $\| u(\cdot, t)\|_{\ld}\leq v(t) \hd \m{for a.a. \hd } t\geq 0$. Applying proposition~\ref{decayode} we finish the proof  of theorem \ref{main}.
\end{proof}

\section{Appendix}

We prove here the estimates similar to Lemma 3.1 in \cite{Zacher}. The final inequality is the same, but in our case the kernel $k(t) = \izj \mg t^{-\al} d\al$ does not satisfy assumptions of Lemma 3.1 in \cite{Zacher}. Although, it is mentioned, that the proof for less regular kernels and more regular function is also possible. However, the authors do not refer to the case of the distributed Caputo derivative.
 Therefore we give here a detailed proof.

\begin{prop}\label{Awniosek}
Assume that $p\in [1,\infty)$, $w$ is continuous on $\overline{\Omega}\times [0,T]$ and
\eqq{\m{if } x\in \Omega, \m{then } w(x,\cdot )\in AC[0,T],  }{A33}
and for fixed $\gamma \in (0,1)$
\eqq{t^{1-\gamma}w_{t}\in L^{\infty}(\Omega^{T})}{A44}
holds. Then for each $t\in (0,T]$ the following estimate
\[
\dm \| w(\cdot , t) \|^{p}_{L^{p}(\Omega)} + \io \izj \al \mg \izt \taj |w(x,t)|^{p}  G\left( \ftt\right) d\tau d\al dx \]
\eqq{+ p \izj \mg t^{-\alpha} d\al \io |w(x,t)|^{p}  G\left( \frac{w(x,0)}{w(x,t)}  \right) dx  = p \io \dm w(x,t) \cdot \fptbt dx }{d}
holds, where $G(s)$ is nonnegative function given by formula $G(s)=\frac{1}{p}|s|^{p}-s+1-\frac{1}{p}$. Furthermore, we have
\eqq{\| w(\cdot, t) \|_{\lp}^{p-1}\dm \nlp{\wct} \leq  \io \dm w(x,t) \cdot \fptbt dx .}{a1}
\label{lpda}
\end{prop}

\begin{proof}
\[
 \left[p \io \dm w(x,t)\cdot \fptbt dx - \dm \| w(\cdot , t) \|_{L^{p}(\Omega)}^{p}  \right]
\]
\[
=p\io \izj \mg \izt \ta w_{\tau }(x,\tau ) \left[ \fptbt - \fptabt \right]  d\tau d\al dx
\]
\[
=p\io \izj \mg |w(x,t)|^{p}\izt \ta \left(  \ftt \right)_{\tau} \left[ 1-\left|\ftt \right|^{p-2} \ftt  \right] d\tau d\al dx.
\]
For $G(s)=\frac{1}{p}|s|^{p}-s+1-\frac{1}{p}$ we have $G'(s)=s|s|^{p-2}-1$, hence we can write
\[
-p\io \izj \mg |w(x,t)|^{p}\izt \ta \left( G\left( \ftt\right) \right)_{\tau} d \tau d\al dx
\]
\[
= - p \lim_{h\rightarrow 0^{+}}  \io \izj \mg |w(x,t)|^{p} \izth \ta \left( G\left( \ftt\right) \right)_{\tau} d \tau d\al dx
\]
\[
= p\lim_{h\rightarrow 0^{+}}  \io \izj \al \mg |w(x,t)|^{p} \izth \taj G\left( \ftt\right)  d \tau d\al dx
\]
\[
- p \lim_{h\rightarrow 0^{+}}  \io \izj \mg |w(x,t)|^{p} \left. \ta  G\left( \ftt\right) \right|_{\tau=0}^{\tau = t-h} d\al dx
\]
\[
=  p\lim_{h\rightarrow 0^{+}} \io \izj \al \mg \izth \taj |w(x,t)|^{p}  G\left( \ftt\right) d\tau d\al dx
\]
\[
+p \izj \mg t^{-\alpha}d\al \io |w(x,t)|^{p} G\left( \frac{w(x,0)}{w(x,t)}  \right)  dx
\]
\[
-p\lim_{h\rightarrow 0^{+}}\izj \mg h^{-\alpha} d\al \io |w(x,t)|^{p} G\left( \frac{w(x,t-h)}{w(x,t)}  \right)dx.
\]
We shall show that the last limit is equal zero. Indeed, by definition we have
\[
 \io |w(x,t)|^{p} G\left( \frac{w(x,t-h)}{w(x,t)}  \right)dx
\]
\[
= \io |w(x,t)|^{p-2}w(x,t)\left[  w(x, t)-w(x,t-h) \right]dx  +\frac{1}{p}\io |w(x,t-h)|^{p}-|w(x,t)|^{p}dx,
\]
hence using the inequality $|a^{p}-b^{p}|\leq p (a^{p-1}+b^{p-1}) |a-b|$, which holds for positive $a,b$ we get
\[
\left|  \izj \mg h^{-\alpha} d\al \io |w(x,t)|^{p} G\left( \frac{w(x,t-h)}{w(x,t)}  \right)dx \right|
\]
\[
\leq 2 \izj \mg h^{-\alpha} d\al \io \left[ |w(x,t)|^{p-1}+|w(x,t-h)|^{p-1} \right] |w(x,t)-w(x,t-h)|dx.
\]
The above expression converges to zero, because $w$ is continuous and by (\ref{A44}) we have
\[
 |w(x,t-h) -w(x,t)|dx\leq  h(t-h)^{\gamma-1} \| t^{1-\gamma} w_{t} \|_{L^{\infty}(\Omega^{T})}.
\]
Thus,
\[
\left|  \izj \mg h^{-\alpha} d\al \io |w(x,t)|^{p} G\left( \frac{w(x,t-h)}{w(x,t)}  \right)dx \right|
\]
\[
\leq c(\Om) \norm{t^{1-\gamma}w_{t}}_{L^{\infty}(\Om^{T})}\norm{w}_{C(\Om^{T})}^{p-1} (t-h)^{\gamma-1} \izj \mg h^{1-\al} d\al \longrightarrow 0 \textrm { as } h \rightarrow 0.
\]
If we note that  $G(s)$ in nonnegative, then using monotone Lebesgue  we get
\[
 \lim_{h\rightarrow 0^{+}} \io \izj \al \mg \izth \taj  |w(x,t)|^{p}  G\left( \ftt\right)d\tau d\al dx
\]
\[
=\io \izj \al \mg \izt \taj |w(x,t)|^{p}  G\left( \ftt\right) d\tau d\al dx
\]
and the proof of (\ref{d}) is finished.

\no To get (\ref{a1}) we write
\[
\dm \nlp{\wct}^{p}- p \nlp{\wct}^{p-1} \dm \nlp{\wct}
\]
\[
=p \izj \mg \izt \ta \left( \nlp{\wcta}^{p-1}- \nlp{\wct}^{p-1} \right) \ddta \nlp{\wcta} \dt d\al
\]
\[
=p \izj \mg \izt \ta \nlp{\wct}^{p} \left[ \left(\frac{  \nlp{\wcta}}{\nlp{\wct}} \right)^{p-1}-1 \right] \ddta   \left(\frac{  \nlp{\wcta}}{\nlp{\wct}} \right)  \dt d\al
\]
\[
=p \izj \mg \izt \ta \nlp{\wct}^{p} \ddta \left[ G \left(\frac{  \nlp{\wcta}}{\nlp{\wct}} \right) \right]  \dt d\al
\]
\[
=- p \izj \al \mg \izt \taj \nlp{\wct}^{p}  G \left(\frac{  \nlp{\wcta}}{\nlp{\wct}} \right)   \dt d\al
\]
\[
+ p \left. \izj \mg  \ta \nlp{\wct}^{p}  G \left(\frac{  \nlp{\wcta}}{\nlp{\wct}} \right) \right|_{\tau =0}^{\tau =t} d\al
\]
\[
=-\izj \frac{\alpha \ma }{\gja} \izt \taj \left( \nlp{\wcta}^{p}- \nlp{\wct}^{p} \right) \dt d\al
\]
\[
+p \frac{\alpha \ma }{\gja} \izt \taj \nlp{\wct}^{p-1} \left( \nlp{\wcta}- \nlp{\wct} \right) \dt d\al
\]
\[
-\izj \frac{\ma}{\gja}  t^{-\alpha} \left( \nlp{w(\cdot, 0)}^{p}- \nlp{\wct}^{p} \right)d\al
\]
\[
+p\izj \frac{\ma  }{\gja} t^{-\alpha} \nlp{\wct}^{p-1} \left( \nlp{w(\cdot , 0)}- \nlp{\wct} \right)d\al
\]
Using the above equality and (\ref{d}) we have
\[
p \io \dm w(x,t) \cdot \fptbt dx = p \nlp{\wct}^{p-1} \dm \nlp{\wct}
\]
\[
+p \izj \frac{\alpha \ma}{\gja} \izt \taj \left[ \nlp{\wcta}\nlp{\wct}^{p-1}- \io  |w(x,t)|^{p-2} w(x,t)w(x,\tau) dx   \right] \dt d\al
\]
\[
+\izj \frac{\ma}{\gja} t^{-\alpha} d\al \left[ \nlp{w(\cdot,0)}\nlp{\wct}^{p-1}- \io  |w(x,t)|^{p-2} w(x,t)w(x,0) dx   \right] .
\]
By H\"older inequality the expressions in square brackets are nonnegative, thus
\[
p \io \dm w(x,t) \cdot \fptbt dx \geq p \nlp{\wct}^{p-1} \dm \nlp{\wct}.
\]

\end{proof}
We quote here theorem 1 from \cite{lmax}. However, in our consideration we need to weaken assumptions from this theorem. Thus, we give a short proof based on the one originally presented in \cite{lmax}.

\begin{lem}\label{luchko}
Let $f: [0,T]\rightarrow \mathbb{R}$ be an absolutely continuous function, which attains its  maximum over the interval $[0,T]$ at the point $t_{0} \in (0,T]$. If for some $\beta \in (0,1]$ and every  $\ka >0$ $f \in W^{1,\frac{1}{1-\beta}}(\ka,T)$,  then  $(D^{\al}f)(t_{0}) \geq 0$ for every $\al \in (0,\beta)$. In particular, if  $f \in W^{1,\infty}(\ka,T)$ for each $\ka>0$, then  $(D^{\al}f)(t_{0}) \geq 0$ for every $\al \in (0,1)$ and $\dm f(t_{0})\geq 0$ for any nonnegative $\mu $.
\end{lem}
\begin{proof}
Firstly, we introduce the function $g(t):=f(t_{0}) - f(t)$ for $t \in [0,T]$.  We  notice that $g(t) \geq 0$ and $(D^{\al}g)(t) = -(D^{\al}f)(t)$ for $t \in [0,T]$. Function $g$ is absolutely  continuous and $g(t_{0})=0$, hence for $\ka\in (0, t_{0})$ we may write
\eqq{
|g(t)| \leq \int^{t_{0}}_{t}\abs{g'(s)}ds \leq \norm{g'}_{L^{\frac{1}{1-\beta}}(\ka,T)}\abs{t-t_{0}}^{\beta} \hd \m{ for } \hd t \in [\ka,t_{0}].
}{nbeta}
Thus for fixed $\al \in (0,\beta)$ we can write that
\[
(D^{\al}g)(t_{0}) = \frac{1}{\Gamma(1-\al)} \int_{0}^{\ka}(t_{0}-\tau)^{-\al}g'(\tau)d\tau + \frac{1}{\Gamma(1-\al)} \int_{\ka}^{t_{0}}(t_{0}-\tau)^{-\al}g'(\tau)d\tau.
\]
We fix  $\ep >0$. Function $g$ is absolutely continuous hence, $g'$ belongs to $L^{1}(0,T)$ and applying Young inequality we deduce that there exists  $\ka>0$ such that the first integral is smaller than $\ep$. To the second one we use integration by parts formula \[
\int_{\ka}^{t_{0}}(t_{0}-\tau)^{-\al}g'(\tau)d\tau = \lim_{h\rightarrow 0^{+}} \int_{\ka}^{t_{0}-h}(t_{0}-\tau)^{-\al}g'(\tau)d\tau
\]
\[
= -(t_{0}-\ka)^{-\al}g(\ka) + \lim_{h\rightarrow 0^{+}} h^{-\al}g(t_{0}-h)- \al\int_{\ka}^{t_{0}}(t_{0}-\tau)^{-\al-1}g(\tau)d\tau \leq 0,
\]
because by (\ref{nbeta})  the limit equals zero. Then $(\da g)(t_{0})\leq \ep $ and $(\da f)(t_{0})\geq -\ep$ for arbitrary $\ep>0$, which finishes the proof.
\end{proof}

\no We recall lemma~6 from \cite{nasza} (see also lemma~2.1 in \cite{Laplace}).

\begin{lem}\label{odwrotna}
Let $F$ be a complex function, satisfying following assumptions:
\begin{enumerate}[1)]
\item
$F(p)$ is analytic in $\mathbb{C} \setminus (-\infty, 0].$
\item
The limit $F^{\pm}(t):=\lim\limits_{\vf\rightarrow \pi^{-}}F(te^{\pm i \vf})$ exists for a.a. $t>0$ and $F^{+} = \overline{F^{-}}$.
\item
For each   $0<\eta<\pi$
\begin{enumerate}[a)]
\item
 $|F(p)| = o(1)$, as $|p|\rightarrow \infty$ uniformly on $ |\arg(p)| < \pi - \eta$
\item
$|F(p)|= o(\frac{1}{|p|})$, as $|p|\rightarrow 0$ uniformly on $|\arg(p)|< \pi - \eta$.
\end{enumerate}
\item
There exists $\ve_{0} \in (0,\frac{\pi}{2})$ and a function $a=a(r)$ such that \hd $\forall  \vf \in (\pi - \ve_{0}, \pi)$  the estimate  $\abs{F(re^{\pm i\vf})}   \leq a(r)$ holds, where $ \frac{a(r)}{1+r} \in L^{1}(\mr_{+})$.
\end{enumerate}
Then for $p\in \mathbb{C}$ such that $\Re{p}>0$ we have
\[
F(p) = \izi e^{-xp}f(x)dx, \textrm{  where  }f(x) = \frac{1}{\pi}\izi e^{-rx} \Im(F^{-}(r))dr.
\]
\end{lem}

\no \textbf{Acknowledgments}. The authors are grateful to Prof. Masahiro Yamamoto for his inspiration to write this paper and for his hospitality during the stay at Tokyo University.

\end{document}